\documentclass[11pt]{article}
\usepackage{latexsym,amssymb,amsmath,amssymb}
\usepackage{amsfonts,amsbsy,amsthm}
\usepackage[all]{xy}
\numberwithin{equation}{section}
\newtheorem{theorem}{Theorem}[section]

\newtheorem{proposition}[theorem]{Proposition}
\newtheorem{corollary}[theorem]{Corollary}

\theoremstyle{definition}
\newtheorem{definition}[theorem]{Definition}

\newtheorem{remark}[theorem]{Remark}

\begin{document}


\topmargin3mm
\hoffset=-1cm
\voffset=-1.5cm
\

\medskip
\begin{center}
{\large\bf The enhanced quotient graph of the quotient of a finite group
}
\vspace{6mm}\\
\footnotetext[2]{This work was partially supported by CONACYT.}

\end{center}

\medskip
\begin{center}
Luis A. Dupont, Daniel G. Mendoza and Miriam Rodr\'{\i}guez.
\\
{\small Facultad de Matem\'aticas, Universidad Veracruzana}\vspace{-1mm}\\
{\small Circuito Gonzalo Aguirre Beltr\'an S/N;}\vspace{-1mm}\\
{\small Zona Universitaria;}\vspace{-1mm}\\
{\small Xalapa, Ver., M\'exico, CP 91090.}\vspace{-1mm}\\
{\small e-mail: {\tt ldupont@uv.mx}\vspace{4mm}}
\end{center}

\medskip

\begin{abstract}
For a finite group $G$ with a normal subgroup $H$, the enhanced quotient graph of $G/H$, denoted by $\mathcal{G}_{H}(G),$ is the graph with vertex set $V=(G\backslash H)\cup \{e\}$ and two vertices $x$ and
$y$ are edge connected if $xH = yH$ or $xH,yH\in \langle zH\rangle$ for some $z\in G$. In this article, we characterize the enhanced quotient graph of $G/H$. The graph $\mathcal{G}_{H}(G)$ is complete if and only if $G/H$ is cyclic, and $\mathcal{G}_{H}(G)$ is Eulerian
if and only if $|G/H|$ is odd. We show some
relation between the graph $\mathcal{G}_{H}(G)$ and the enhanced power graph $\mathcal{G}(G/H)$ that was introduced by Sudip Bera and A.K. Bhuniya (2016).  The graph $\mathcal{G}_H(G)$ is complete if and only if $G/H$ is cyclic if and only if $\mathcal{G}(G/H)$ is complete. The graph $\mathcal{G}_H(G)$ is Eulerian if and only if $|G|$ is odd if and only if $\mathcal{G}(G)$ is Eulerian, i.e., the property of being Eulerian does not depend on the normal subgroup $H$.

\end{abstract}

\section{Introduction}

The investigation of graphs related to groups as
well as other algebraic structures is an exciting research topic in
the last few decades, ~\cite{Abdollahi,Abe,AndersonLivings,Atani,Ballester1,Ballester2,Chakrabarty,Elavarasan,Herzog,Redmond,Willians,Wu}. Only basic concepts about graphs will be needed for
this paper. They can be found in any book about graph theory, for example ~\cite{Diestel}.
\\
Given a finite group $G$, there are many ways to associate a graph to $G$ by
taking families of elements or subgroups as vertices and letting two vertices be
joined by an edge if and only if they satisfy a property. All groups in this paper are finite. The \emph{enhanced power graph} of a group was introduced by Sudip Bera and A. K. Bhuniya~\cite{Bera-Bhuniya}.
\begin{definition}
Given a finite group $G$, the \emph{enhanced power graph} of $G$ denoted by $\mathcal{G}(G),$ is the graph with vertex
set $V(\mathcal{G}(G))=G$ and two distinct vertices $x, y$ are edge connected $\{x,y\}$ in $E(\mathcal{G}(G))$  if there exists $z \in G$ such that
$x,y \in \langle z\rangle$.
\end{definition}
We will apply the idea of studying the enhanced power graph through the graphs of the quotient groups, as was carried out by the authors for the normal subgroup based power graph of a finite group in ~\cite{Bera-Bhuniya2}.
\begin{definition}
For a finite group $G$ with a normal subgroup $H$, the \emph{enhanced quotient graph} of $G/H$, denoted by $\mathcal{G}_{H}(G),$ with vertex set $V(\mathcal{G}_{H}(G))=(G\backslash H)\cup \{e\}$ and two vertices $x$ and
$y$ are edge connected, (i.e. $\{x,y\}\in E(\mathcal{G}_{H}(G))$), if $xH = yH$ or $xH,yH\in \langle zH\rangle$ for some $z\in G$.
\end{definition}

Let $G$ be a finite group of order $n$ and $H$ be a normal subgroup of $G$ with $|H|=m$. In this paper we provide some results on the basic structure of the enhanced quotient graph and we show the interplay between the graph $\mathcal{G}_{H}(G)$ and the enhanced power graph $\mathcal{G}(G/H)$.

\section{Definitions and structure}
In this section we provide the first results of the two graphs of groups that we will study in this article, the \emph{enhanced power graph} of $G$ denoted by $\mathcal{G}(G)$ and the \emph{enhanced quotient graph} of $G/H$, denoted by $\mathcal{G}_{H}(G)$. We will see the results of the graph $\mathcal{G}_{H}(G)$ in comparison to the graph $\mathcal{G}(G/H)$ which was introduced in~\cite{Bera-Bhuniya}.

\begin{proposition}
$\mathcal{G}_H(G)$ is connected.
\end{proposition}

\begin{proof}
Since $x^nH = eH$ and consequently $xH, eH \in \langle xH \rangle,$ $\{ x, e\}$ is and edge of  $\mathcal{G}_H(G)$ for all $x \in G \setminus H.$
\end{proof}

\begin{proposition}
For all $g \in G,$ $gH \cap V(\mathcal{G}_H(G))$ is a clique of the graph $\mathcal{G}_H(G).$
\end{proposition}

\begin{proof}
Let $gh_1$ and $gh_2$ be elements of $gH \cap V(\mathcal{G}_H(G)),$ then $gh_1 =gH =gh_2H$ hence $\{gh_1, gh_2\} \in E_h.$
\end{proof}

\begin{corollary}
The graph $\mathcal{G}_H(G)$ contains at least $[G:H]-1$ isomorphic subgraphs to $K_m.$
\end{corollary}

\begin{proposition}\label{propoclique}
If $aH, bH \in G/H$ with $aH \neq bH,$ $a,b \notin H$ and some element of $aH$ forms an edge of
some element with $bH,$ then $aH \cup bH$ is a clique of $\mathcal{G}_H(G).$
\end{proposition}

\begin{proof}
Suppose that $ah_1 \in aH, \; bh_2 \in bH$ with $\{ah_1, bh_2\} \in E(\mathcal{G}_H(G)).$ Then,
$ah_1 = aH, bh_2H = bH \in \langle zH \rangle$ for some $z \in G,$ that is,
$aH = z^rH$ and $bH = z^sH$ for all $ah' \in aH$ and for all $bh'' \in bH$ $ah'H, bh''H \in \langle zH \rangle.$
Therefore, $\{ah', bh''\} \in E(\mathcal{G}_H(G)).$
\end{proof}

\begin{corollary}
The graph $\mathcal{G}_H(G)$ contains at least $|E(\mathcal{G}(G/H))|-[G:H]+1$ isomorphic subgraphs to $K_{2m}.$
\end{corollary}
\begin{definition}
A \emph{clique} $K$, in a graph $\mathcal{F} = (V, E)$ is a subset of the vertices of $\mathcal{F}$ such that every two distinct vertices are adjacent. This is equivalent to the condition that the induced subgraph of $\mathcal{F}$ induced by $K$ is a complete graph. A \emph{maximum clique} of a graph is a clique such that there is no clique with more vertices. The \emph{clique number} $\omega(\mathcal{F})$ of a graph $\mathcal{F}$ is the number of vertices in a maximum clique in $\mathcal{F}$.
\end{definition}
\begin{corollary}\label{paraplanas}
For $\mathcal{G}_H(G)$, $\omega(\mathcal{G}_H(G))=|H|^{s-1}+1$, where $s$ is the highest order of a cyclic subgroup of $G/H$.
\end{corollary}
\begin{proof}
It follows from the proof of Proposition~\ref{propoclique}.
\end{proof}

\begin{proposition}
Let $a,b$ be elements of $V(\mathcal{G}_H(G)).$ Then, $\{a,b\} \in E(\mathcal{G}_H(G))$ if and only if $aH = bH$ or
$\{ aH, bH\} \in E(\mathcal{G}(G/H)).$
\end{proposition}

\begin{proof}
If $\{a,b\} \in E(\mathcal{G}_H(G)),$ then $aH, bH \in \langle zH \rangle$ for some $z \in G.$
Then, if $aH \neq bH$ we would have that $\{aH, bH \} \in E(\mathcal{G}(G/H)).$
On one hand, $\{aH, bH\} \in E(\mathcal{G}(G/H))$ implies that $aH, bH \in \langle zH \rangle$ for some
$z \in G,$ thus $\{a, b\} \in E(\mathcal{G}_H(G)).$ On the other hand, $aH = bH$ implies that $aH, bH \in \langle aH \rangle.$
Hence, $\{a,b\}\in E(\mathcal{G}_H(G)).$
\end{proof}

\begin{corollary}
The graph $\mathcal{G}_H(G)$ contains at least $|H|^{[G:H]-1}$ isomorphic subgraphs to $\mathcal{G}(G/H).$
\end{corollary}
\begin{definition}
For a finite group $W$ and $w$ an element of $W$, we denote by $\mathcal{G}_{en}(w)$ the set of generators of $\langle w \rangle$.
\end{definition}

\begin{proposition}\label{teorema2.1GPE} \emph{\cite{Bera-Bhuniya}}
For $a,b \in G$ with $|a|=|b|$ and $\langle a \rangle \neq \langle b\rangle$
we have that none of the vertices in $\mathcal{G}_{en}(a)$ is adjacent with vertices in
$\mathcal{G}_{en}(b),$ in the graph of $\mathcal{G}(G).$
\end{proposition}

\begin{proposition}
For $a,b \in G\setminus H$ with $|aH|=|bH|$ and $\langle aH \rangle \neq \langle bH\rangle$
we have that none of the vertices  in $\mathcal{G}_{en}(aH)$ is adjacent with vertices in
$\mathcal{G}_{en}(bH),$ in the graph of $\mathcal{G}(G/H)$ and the graph $\mathcal{G}_H(G).$
\end{proposition}
\begin{proof}
For the graph $\mathcal{G}(G/H)$ it follows from Proposition~\ref{teorema2.1GPE}.
For $\mathcal{G}_H(G)$ we see that if $xH \in \mathcal{G}_{en} (\langle aH\rangle)$ and $yH \in \mathcal{G}_{en} (\langle bH\rangle)$ are adjacent in $\mathcal{G}_H(G),$ then $\langle xH \rangle  \langle yH\rangle \subseteq \langle zH\rangle $ for some $z \in G.$ Then $\langle aH \rangle , \langle bH \rangle < \langle zH \rangle$ with $|\langle aH \rangle|=|\langle bH \rangle|$ and $\langle aH \rangle \neq \langle bH \rangle,$
a contradiction.
\end{proof}

\begin{proposition}\label{teorema2.2GPE} \emph{\cite{Bera-Bhuniya}}
The graph $\mathcal{G}(G)$ contains a cycle if and only if $|a| \geq 3,$
for some $a \in G.$
\end{proposition}

\begin{proposition}
The graph $\mathcal{G}_H(G)$ contains a cycle if and only if $|aH| \geq 3,$ for some $a \in G.$
\end{proposition}
\begin{proof}
Suppose that $x_1H \sim x_2H \sim \cdots \sim x_{n-1}H \sim x_nH= x_1H$ is a cycle with $n \geq 3$
and $x_iH \neq x_jH$ for $i \neq j$ in $\{1, \ldots, n-1\}.$
Thus $x_iH \neq x_jH$ implies that there is $zH$ such that $x_1H, x_2H \in \langle zH \rangle.$ Hence,
$|zH| \geq 3.$
The converse is trivial.
\end{proof}

\begin{corollary}\label{corollarybipartita}
Let $G$ be a group and let $H$ be a normal subgroup of $G.$ Then the following conditions are equivalent.
\begin{enumerate}
\item $\mathcal{G}_H(G)$ is bipartite;
\item $\mathcal{G}(G/H)$ is bipartite;
\item $\mathcal{G}_H(G)$ is a tree;
\item $\mathcal{G}(G/H)$ is a tree;
\item $G/H \cong \mathbb{Z}_2 \times \mathbb{Z}_2 \times \cdots \times \mathbb{Z}_2;$
\item $\mathcal{G}_H(G)$ is a star graph;
\item $\mathcal{G}(G/H)$ is a star graph.
\end{enumerate}
\end{corollary}

\section{Completeness}
In this section we characterize when the graphs studied are complete graphs.

\begin{theorem}
The graph $\mathcal{G}_H(G)$ is complete if and only if $G/H$ is cyclic.
\end{theorem}
\begin{proof}
Suppose that $G/H = \langle xH \rangle.$ Let $a,b$ vertices of $\mathcal{G}_H(G).$
Then, $aH, bH \in G/H = \langle xH \rangle.$ Thus, $\{a, b\} \in E(\mathcal{G}_H(G))$ and therefore $\mathcal{G}_H(G)$ is complete.
Conversely, suppose that $\mathcal{G}_H(G)$ is complete. Let us consider that $xH \in G/H$ with $|xH| = \max \{ |gH|: g \in G\}.$ We claim that $G/H = \langle xH \rangle.$
Let $aH \in G/H.$ If $\mathcal{G}_H(G)$ is complete, then $aH, xH \in \langle zH \rangle$ for some $z \in G.$
Thus, $\langle aH\rangle < \langle zH\rangle$ and $\langle xH \rangle < \langle zH \rangle.$ By maximality
$\langle xH \rangle = \langle zH\rangle$ and $aH \in \langle xH \rangle.$ Therefore $G/H = \langle xH\rangle.$
\end{proof}

\begin{corollary} \emph{\cite{Bera-Bhuniya}}
$\mathcal{G}(G)$ is complete if and only if $G$ is cyclic.
\end{corollary}
\begin{proof}
Take $H=\{e\}.$
\end{proof}

\begin{corollary}
$\mathcal{G}_H(G)$ is complete if and only if $G/H$ is cyclic if and only if $\mathcal{G}(G/H)$ is complete.
\end{corollary}

\begin{definition}
Define a set $C \subseteq G$ as invertible if is closed under inverses and $e \notin C.$
\\
The Cayley's graph $\mathcal{C}(G, C)$ is defined by:
\[ V_{G,C} := V(\mathcal{C}(G,C)):= G\]
where $\{g, h\} \in E(\mathcal{C}(G,C))$ if and only if $hg^{-1} \in C.$
\end{definition}

We can summarize the theme of the graphs in the following result.
\begin{theorem}
We have that
\begin{itemize}
\item[(i)] $\mathcal{G}_H(G) \dot\cup K_H = \mathcal{C}(G, H\setminus \{e\}) \dot\cup \langle G\setminus H \cup \{e\}\rangle_{\mathcal{G}(G/H)}$ where $\langle G\setminus H \cup \{e\} \rangle_\mathcal{G}(G/H)$ is
the graph induced by $G \setminus H \cup \{e\}$ in $\mathcal{G}(G/H).$
\item[(ii)] The following propositions are equivalent
\begin{enumerate}
\item $\mathcal{G}_H(G) \dot\cup K_H \cong e \ast (K_{|G-H|} \dot\cup K_{|H-\{e\}|})$
\item $G$ is cyclic.
\item $\mathcal{G}_H(G)$ is $|G-H|$-regular.
\item $\mathcal{G}(G/H)$ is complete.
\item $\mathcal{G}(G/H)$ is $([G:H]-1)$-regular.
\end{enumerate}
\end{itemize}
\end{theorem}

\section{Cone Property}

\begin{definition}
We say that a vertex $v$ of a graph $\mathcal{G}$ is a \emph{cone vertex} if
$\{v, w\}$ is a edge of the graph for each vertex $w$ of $\mathcal{G}.$ A graph having a cone vertex is said to satisfy the \emph{cone property}.
\end{definition}
\begin{remark} For all finite groups $K, G$ we have that $\mathcal{G}_{ \{e_G\}\times K} (G \times K) \not\cong \mathcal{G}(G)= \mathcal{G}(G \times K / K)$ since one graph has more
vertices than the other.
\end{remark}
\begin{theorem}\label{cono1}\emph{\cite{Bera-Bhuniya}}
Let $G$ be a finite group and $n\in \mathbb{N}$. If $gcd(|G|,n)=1$, then  $\mathcal{G}(G\times \mathbb{Z}_n)$ has a cone vertex.
\end{theorem}
\begin{theorem}\label{cono-vertex-relativo}
Let $G$ be a group and let $H$ be a normal subgroup of $G.$ If $\mathcal{G}_{H \times \{\bar{0}\}} (G\times \mathbb{Z}_n)$ with
$(|G|, n)=1,$ then each $(e,a)$ with $(a,n)=1$ is a cone vertex.
\end{theorem}
\begin{proof}
The argument is similar to the proof of the Theorem~\ref{cono1}.
\end{proof}

\begin{theorem}\label{cono2}\emph{\cite{Bera-Bhuniya}}
Let $G$ be a finite abelian group. Then
$\mathcal{G}(G)$ has a cone vertex if and only if $G$ has a cyclic Sylow $p$-subgroup.
\end{theorem}
\begin{theorem}
Let $G$ be a finite abelian group, with $|H|=m$, $|G|=mp^s$ and $gcd(m,p)=1$. Then
$\mathcal{G}_H(G)$ has a cone vertex if and only if $G$ has a cyclic Sylow $p$-subgroup.
\end{theorem}

\begin{proof}
If $G=H\times \mathbb{Z}_{p^s},$ then $\mathcal{G}_H(G)= \mathcal{G}_{H\times \{G\}}(H \times \mathbb{Z}_{p^s}).$ By Theorem~\ref{cono-vertex-relativo} $(e, \bar{1})$ is a cone vertex.

Conversely, suppose that $G$ has not the cyclic Sylow $p$-subgroup, we can assume that $s \geq 2.$
$G = H \times \mathbb{Z}_{p^{s_1}} \times \mathbb{Z}_{p^{s_2}} \times \cdots \times \mathbb{Z}_{p^{s_k}}.$

Let $v= (h_1, x_1, x_2, \ldots, x_k)$  be the cone vertex and consider $w= (e, 0, \ldots, 0, \bar{1})$
with $wH$ element of maximum order in $G/H$. So if $v$ is a cone vertex, then
$v \sim w,$ which implies that $vH, wH \in \langle zH \rangle$ for some $p$-element $z.$
By maximality, $vH \in \langle wH \rangle = \langle zH \rangle.$
Then, $v \not\sim a_1 = (e, x_1, 0, \ldots, 0, 0)$ with $|a_1H| = p \mid |vH|,$
otherwise, it would be $z$ such that $\langle a_1H \rangle, \langle vH \rangle < \langle zH\rangle$
concluding that $\langle a_1H \rangle < \langle vH\rangle$ and similarly $a_1H \neq a_2H$
with $a_2 = (e, 0, x_2, 0, \ldots, 0)$ implies that $\langle a_2H \rangle < \langle vH\rangle$
with $|a_2H|=p,$ a contradiction.
\end{proof}

\begin{theorem}\label{cono3}\emph{\cite{Bera-Bhuniya}}
Let G be a non-abelian $p$-−group. $\mathcal{G}(G)$ satisfies the
cone property if and only if $G$ is generalized quaternion group.
\end{theorem}

\begin{theorem}
Let $H$ be a normal subgroup of  $G,$ where $G$ is a $p$-group and $|H|=p.$ Then
$\mathcal{G}_H(G)$ has the cone property if and only if $G/H$ is a generalized quaternion $p$-group.
\end{theorem}
\begin{proof}
Suppose that $G$ is a generalized quaternion $p$-group. Thus, $\langle xH\rangle$ is the only subgroup of order $p.$
Let $z\in G\setminus H,$ then $|zH|=p^t$ with $t \geq 1.$ Hence, $\langle xH\rangle < \langle zH\rangle.$ Thus $x \sim z$ and therefore $x$ is a cone element.
Conversely, let $e \neq X$ a cone element de $\mathcal{G}_H(G),$
for all $g \in G\setminus H$ with $|gH|=p,$ $x\sim g,$ i.e., $\langle xH \rangle, \langle gH\rangle < \langle zH \rangle.$ Therefore,  $\langle gH \rangle < \langle xH\rangle.$ Henceforth each element of
$G/H$ of order $p$ belongs to a cyclic group. We conclude that $G/H$ has an unique subgroup of order $p.$
\end{proof}

\begin{theorem}\label{cono4}\emph{\cite{Bera-Bhuniya}}
Let $G$ be any simple group. Then $\mathcal{G}(G)$ does not satisfy the cone property.
\end{theorem}

\begin{theorem}
Let $H$ be a maximal normal subgroup of $G$. Then $\mathcal{G}_H(G)$ does not satisfy
the cone property.
\end{theorem}
\begin{proof}
Suppose that $x$ is a cone element, $|Hx|=m$ and $p$ is  a prime such that $p\mid m.$
Let $zH$ with $|zH|= p$. Since $x\sim z$, $\langle xH \rangle, \langle zH\rangle < C$ cyclic implies $\langle zH \rangle < \langle xH \rangle.$ Therefore there exists a unique subgroup of order $p$
in the not simple group $G,$ a contradiction.
\end{proof}

\section{The properties: Eulerian, Hamiltonian and planar}
A closed walk in a graph $\mathcal{H}$ containing all the edges of $\mathcal{H}$ is called an \emph{Euler path} in $\mathcal{H}$. A
graph containing an Euler path is called an \emph{Euler graph} or \emph{Eulerian graph}. The following theorem due to Euler \cite{Euler},
characterises Eulerian graphs. Euler proved
the necessity part and the sufficiency part was proved by Hierholzer, \cite{Heirholzer}.

\begin{theorem}\emph{ (Euler) } A connected graph $\mathcal{H}$ is an Euler graph if and only if all vertices
of $\mathcal{H}$ are of even degree.
\end{theorem}

\begin{theorem}\emph{\cite{Bera-Bhuniya}}
Let $G$ be a group of order $n$. Then the enhanced power graph $\mathcal{G}(G)$ is Eulerian if
and only if $n$ is odd.
\end{theorem}

\begin{theorem}
The graph $\mathcal{G}_H(G)$ is Eulerian if and only if $|G|$ is odd.
\end{theorem}
\begin{proof}
If $|G|$ is odd, then $\mathcal{G}_H(G)$ is Eulerian. If $\{ \overline{C_1}, \ldots, \overline{C_s} \}$ are the
maximal cyclic subgroups  of $G/H$
that contain $\langle gH \rangle$, with $g \in G\setminus H$. Then
$$\deg (g) = [(|C_1|-1)|H|+1]+ \sum_{i=2}^s|\overline{C_i}-A_i||H|,$$ where $A_{i}$ is the maximum of $C_{i}$ that is contained in some
$C_{j}$, for $j=2,3,\ldots,(i-1)$. Thus, $\deg (g)$ is even and $\deg (e) = |G|-|H|$ is even. \\

Suppose that $\mathcal{G}_H(G)$ is Eulerian, consequently $\deg(e) = |G|-|H|$ is even.
Therefore $|G|$ and $|H|$ are both even or they are both odd. \\

Suppose that $|G|$ and $|H|$ are both even, then for $g \in G\setminus H,$
$$\deg (g) = [(|C_1|-1)|H|+1]+ \sum_{i=2}^s|\overline{C_2}-A_2||H|$$
where $(|C_1|-1)|H|$ and $\sum_{i=2}^s|\overline{C_2}-A_2||H|$ are both even. Therefore, the degree of $g$ is odd, a contradiction.
\end{proof}

We conclude that the property of being Eulerian does not depend on the normal subgroup $H$.

\begin{corollary}
$\mathcal{G}_H(G)$ is Eulerian if and only if $|G|$ is odd if and only if $\mathcal{G}(G)$ is Eulerian.
\end{corollary}

\begin{definition}
A graph $\mathcal{H}$ is called \emph{Hamiltonian} if it has a cycle that meets every vertex. Such a cycle is called
a \emph{Hamiltonian cycle}.
\end{definition}

\begin{theorem}
Let $G$ be a finite group and $H$ be a normal subgroup. Then the enhanced quotient graph $\mathcal{G}_H(G)$ is Hamiltonian if
the enhanced power graph $\mathcal{G}(G/H)$ is Hamiltonian.
\end{theorem}
\begin{proof}
Let the enhanced power graph $\mathcal{G}(G/H)$ be Hamiltonian. Suppose $H = \{h_{1},h_{2},\ldots,h_{m}\}$. Since $\mathcal{G}(G/H)$ is Hamiltonian, there exists a Hamiltonian cycle $C = H \sim a_{1}H \sim a_{2}H \sim \cdots\sim a_{n}H \sim H$. So, $a_{i}H\sim a_{j}H$ implies that $a_{i}h_{r}\sim a_{j}h_{s}$ for all $r,s\in\{1,2,\ldots,m\}$ and every coset $a_{i}H$ is a clique in $\mathcal{G}_H(G)$. Hence, we can construct a Hamiltonian cycle in $\mathcal{G}_H(G)$ as follows:
$$e \sim a_{1}h_{1}\sim a_{1}h_{2} \sim\cdots\sim a_{1}h_{m}$$$$a_{1}h_{m}\sim a_{2}h_{1}\sim a_{2}h_{2}\sim\cdots\sim a_{2}h_{m}$$$$a_{2}h_{m}\sim\cdots\sim a_{n}h_{1}\sim a_{n}h_{2}\sim\cdots\sim a_{n}h_{m}\sim e.$$
\end{proof}

\begin{definition}
A \emph{planar graph} is a graph that can be embedded in the plane, i.e., it can be drawn on the plane in such a way that its edges intersect only at their endpoints.
\end{definition}

\begin{theorem} \emph{(Kuratowski 1930, ~\cite{Kuratowski})}
A finite graph is planar if and only if it does not contain a subgraph that is homeomorphic to $K_{5}$ or $K_{3,3}$.
\end{theorem}

\begin{theorem}\emph{\cite{Bera-Bhuniya}}
Let $G$ a group. Then the enhanced power graph
$\mathcal{G}(G)$ is planar if and only if $\Pi_{\epsilon}(G) \subset \{1, 2, 3, 4\}.$
\end{theorem}

\begin{theorem}
The graph $\mathcal{G}_H(G)$ is not planar if and only if $|H|\geq4$ or $|H|=1$ and $s\geq5,$ or $|H|=2$ and $s\geq3,$ or $|H|=3$ and $s\geq2$,
where $s$ is the highest order of a
cyclic subgroup of $G/H$.
\end{theorem}
\begin{proof}
It follows from the Proposition~\ref{propoclique} and Corollary~\ref{paraplanas}.
\end{proof}

\section{Deleted enhanced quotient graph}
\begin{definition}
In this section we consider the graphs $\mathcal{G}^\ast (G),\mathcal{G}_H^\ast (G)$ obtained by deleting the vertex $e$ from the graphs
$\mathcal{G}(G),\mathcal{G}_H(G)$, respectively. We call $\mathcal{G}^\ast (G),\mathcal{G}_H^\ast (G)$ the \emph{deleted enhanced power graph} and the \emph{deleted enhanced quotient graph}. The deleted graphs are not necessarily connected. \cite{Bera-Bhuniya}.
\end{definition}

\begin{theorem}\label{BBT5.1}\emph{\cite{Bera-Bhuniya}}
Let $G$ be a finite $p$--group. Then $\mathcal{G}^\ast(G)$ is connected if and only if $G$ has a unique minimal subgroup.
\end{theorem}

\begin{theorem}
Let $G/H$ be a finite $p$--group. Then $\mathcal{G}_H^\ast (G)$ is connected if and only if $G/H$ has a unique minimal subgroup.
\end{theorem}
\begin{proof}
The argument is similar to the proof of the Theorem~\ref{BBT5.1}.
\end{proof}

For a finite group G, you have the following definitions:
$$\Pi_{\epsilon}(G)=\{|x| : x\in G\}, \; \Pi(G)=\{p\in\mathbb{N} : p\mid |G|\textmd{ is a prime }\},$$
and let $\mu(G)\subseteq\Pi_{\epsilon}(G)$ be the set of all maximal element of $\Pi_{\epsilon}(G)$ under the divisibility relation.

\begin{theorem}\emph{\cite{Bera-Bhuniya}}
Let $|\Pi(Z(G))|\geq 2$. Then the graph $\mathcal{G}^\ast (G)$ is connected.
\end{theorem}

\begin{theorem}
For $Z:= Z(G/H)$ with $|\Pi (Z)| \geq 2$, $\mathcal{G}_H^\ast (G)$ is connected.
\end{theorem}
\begin{proof}
If $Z$ is abelian then $|\mu (Z)|=1,$ $\mu (Z) = \{t\}.$ We consider $gH \in Z$ such that $|gH|=t.$
We claim that for each $x \in G\setminus H,$ $x \sim g.$
\begin{itemize}
\item[Case 1] $\Pi (|xH|) = \Pi (t)$ \\
$|\Pi (Z)| \geq 2$ implies that there are $r, s \in \mathbb{N}$ and prime numbers $p, q$
such that $|x^rH =p|$ and $|g^s|=q$ \\
$(p,q)=1$ and $gH \in Z$ implies that $|x^rg^sH|=pq$ \\
$x^r H, g^s H \in \langle x^r g^s H \rangle.$ Therefore $x \sim x^r \sim g^s \sim g.$
\item[Case 2] $\Pi (|xH|) \neq \Pi (t)$ \\
Thus there is $p \in \Pi(|xH|) \setminus \Pi (t)$ or $p \in \Pi(t) \setminus \Pi (|xH|).$
\begin{itemize}
\item[subcase] $p \in \Pi (|xH|)\setminus \Pi (t)$ \\
Let $q \neq p$ be primes with $q \in \Pi (t),$ with a similar argument we have that $x\sim x^r \sim g^s \sim g.$ The other subcase is similar.
\end{itemize}
\end{itemize}
\end{proof}

A last result of the same nature, concerning the deleted enhanced power graph $\mathcal{G}^\ast(G),$ which has its analogue for the deleted enhanced quotient graph $\mathcal{G}_{H}^\ast(G)$ is the following:

\begin{theorem}\label{BBT5.3}\emph{\cite{Bera-Bhuniya}}
Let $|\Pi(G)|\geq2$ $(|\Pi(G/H)|\geq 2)$ and $|Z(G)|=p^{t}$ $(|Z(G/H)|=p^{t})$ for some prime $p\in\Pi(G)$ $p(\in\Pi(G/H))$. Then the graph $\mathcal{G}^\ast(G)$ $(\mathcal{G}_{H}^\ast(G))$ is
connected if and only if for some non--central element $x$ of order $p$ there exists a non $p$--element $g$
such that $x \sim g$ in the graph $\mathcal{G}^\ast(G)$ $(\mathcal{G}_{H}^\ast(G))$.
\end{theorem}

Finally, the Corollary~\ref{corollarybipartita} for deleted graphs can be rewritten as follows:
\begin{theorem}
Let $G$ be a group. Then the following conditions are equivalent.
\begin{enumerate}
\item $\mathcal{G}^\ast(G)$ is bipartite;
\item $\mathcal{G}^\ast(G)$ is a forest;
\item $\mathcal{G}^\ast(G)$ has no cycle;
\item $|g|\leq3$ for every $g \in G$.
\end{enumerate}
\end{theorem}

The \emph{circumference} of a graph is the length of the longest cycle.
\begin{theorem}
Let $G$ be a group and let $H$ be a normal subgroup of $G$, with $G\backslash H=\dot\cup_{i=1}^{k}g_{i}H$. Then the following conditions are equivalent.
\begin{enumerate}
\item $\mathcal{G}_H^\ast(G)$ is $k$-partite with partition $V=V_{1}\cup V_{2}\cup\cdots\cup V_{k}$; where $|V_{i}\cap g_{j}H|=1$ for all $i,j$.
\item The circumference of $\mathcal{G}_H^\ast(G)$ is $|H|$;
\item The clique number $\omega(\mathcal{G}_H^\ast(G)$) of $\mathcal{G}_H^\ast(G)$ is $|H|$;
\item $|aH|\leq2$ for every $aH \in G/H$;
\item $G/H \cong \mathbb{Z}_2 \times \mathbb{Z}_2 \times \cdots \times \mathbb{Z}_2$.
\end{enumerate}
\end{theorem}



\begin{thebibliography}{10}


\bibitem{Abdollahi} A. Abdollahi, S. Akbari, and H. R. Maimani. \newblock Non-commuting graph of a group. \newblock {\em J.
Algebra}. 298(2):468–-492, 2006.

\bibitem{Abe} S. Abe and N. Iiyori. \newblock A generalization of prime graphs of finite groups.\newblock {\em Hokkaido Math.
J.}, 29(2):391–-407, 2000.

\bibitem{AndersonLivings} D. F. Anderson, P. S, Livingston. \newblock  The zero-divisor graph of a commutative ring.
\newblock {\em J. Algebra.} 217:434--447, 1999.


\bibitem{Atani} S. E. Atani.\newblock  A ideal based zero divisor graph of a commutative semiring.\newblock {\em Glasnik Matematicki}. 44(64):141--153, 2009.

\bibitem{Ballester1} A. Ballester-Bolinches and J. Cossey. \newblock Graphs, partitions and classes of groups.\newblock {\em Monatsh.
Math.}, 166:309–-318, 2012.

\bibitem{Ballester2} A. Ballester-Bolinches, J. Cossey, and R. Esteban-Romero. \newblock On a graph related to permutability
in finite groups. \newblock {\em Ann. Mat. Pura Appl.}, 189(4):567–-570, 2010.

\bibitem{Bera-Bhuniya} S. Bera, A. K. Bhuniya. \newblock On some properties of enhanced power graph.  \newblock {\em arXiv:1606.03209v1}, 2016.

\bibitem{Bera-Bhuniya2} S. Bera, A. K. Bhuniya. \newblock Normal subgroup based power graph of a finite Group.  \newblock {\em Communications in Algebra}, 45 (8): 3251--3259, 2017.

\bibitem{Chakrabarty} I. Chakrabarty, S. Ghosh, M. K. Sen. \newblock Undirected power graphs of semigroups. \newblock {\em Semigroup Forum}. 78:410--426, 2009.
\bibitem{Diestel} R. Diestel. \newblock Graph theory. \newblock {\em volume 173 of Graduate Texts in Mathematics. Springer-Verlag,
Berlin, third edition}, 2005.

\bibitem{Elavarasan} B. Elavarasan, K. Porselvi. \newblock An ideal based zero divisor graph of posets. \newblock {\em Commun.
Korean Math. Soc.} 28:79--85, 2013.

\bibitem{Euler} L. Euler. \newblock Solutio problematics ad geometriamsitus pertinents.\newblock {\em Comment. Academiae
Sci. I Petropolitanae} 8 128-–140, (1736).

\bibitem{Heirholzer} C. Heirholzer. \newblock Ueber die Moglickkeit, Einen Linienzug ohne Wiederholung and
ohne Unterrechung zu umfahren. \newblock {\em Math. Ann.} 6 30-–32, (1893).

\bibitem{Herzog} M. Herzog, P. Longobardi, and M. Maj.\newblock  On a commuting graph on conjugacy classes of
groups.\newblock {\em Comm. Algebra}, 37(10):3369-–3387, 2009.

\bibitem{Kuratowski} K. Kuratowski, Kazimierz \newblock Sur le probl$\mathrm{\grave{e}}$me  des courbes gauches en topologie. \newblock
             Fund. Math. (in French), 15: 271–-283, 1930.

\bibitem{Redmond} S. P. Redmond.\newblock  An ideal-based zero divisor graph of a commutative ring. \newblock {\em Communication
in algebra}. 31:4425--4443, 2003.

\bibitem{Willians} J. S. Williams.\newblock  Prime graph components of finite groups. \newblock {\em J. Algebra}, 69:487–-513, 1981.

\bibitem{Wu} Y. F. Wu. \newblock  Groups in which commutativity is a transitive relation. \newblock {\em J. Algebra}, 207:165–-181, 1998.
\end{thebibliography}

\end{document}